\documentclass[a4paper,reqno]{amsart}
\thanks{For the purpose of open access, the author has applied a CC BY public copyright licence to any Author Accepted Manuscript version arising.
\newline
\indent Data Access Statement: Data sharing is not applicable to this article as no datasets were generated or analysed in this research.
\newline
\indent The author wishes to thank Nigel Byott for comments and suggestions on earlier versions of this paper, as well as for much fruitful discussion surrounding the ideas presented in this work.}
\title[Parallel HGS]{Hopf--Galois structures on parallel extensions}
\author{Andrew Darlington}
\date{\today}
\address{Department of Mathematics and Statistics, Faculty of Environment, Science and Economy, University of Exeter, Exeter EX4 QFU. UK.}
\email{ad788@exeter.ac.uk}

\usepackage[utf8]{inputenc}
\usepackage{amsmath}
\usepackage{amsfonts}
\usepackage{amssymb}
\usepackage{amsthm}
\usepackage{pifont} 
\usepackage{enumerate} 
\usepackage{enumitem}
\usepackage{wasysym} 
\usepackage{mathrsfs}
\usepackage{xfrac} 
\usepackage{array}
\usepackage{siunitx}
\usepackage{mathtools}
\usepackage{framed}
\usepackage{pifont}
\usepackage[utf8]{inputenc}
\usepackage[english]{babel}
\usepackage[utf8]{inputenc}
\usepackage{tikz,tkz-euclide}
\usepackage{xcolor,graphicx}
\usepackage{faktor}
\usepackage{xfrac}
\usepackage{float}
\usepackage{tikz-cd}
\usepackage{rotating}
\usepackage{hyperref}
\usepackage{hhline}

\makeatletter
\def\bign#1{\mathclose{\hbox{$\left#1\vbox to8.5\p@{}\right.\n@space$}}\mathopen{}}
\def\Bign#1{\mathclose{\hbox{$\left#1\vbox to11.5\p@{}\right.\n@space$}}\mathopen{}}

\newtheorem{theorem}{Theorem}[section]
\newtheorem{proposition}[theorem]{Proposition}
\newtheorem{lemma}[theorem]{Lemma}
\newtheorem{conjecture}[theorem]{Conjecture}
\newtheorem{corollary}[theorem]{Corollary}

\theoremstyle{definition}
\newtheorem{remark}[theorem]{Remark}

\newcommand{\e}{\textbf{e}}

\newcommand{\Gal}{\mathrm{Gal}}
\newcommand{\Hol}{\mathrm{Hol}}
\newcommand{\Aut}{\mathrm{Aut}}

\begin{document}
	\bibliographystyle{amsalpha}
	
\begin{abstract}
    Let $L/K$ be a finite separable extension of fields of degree $n$, and let $E/K$ be its Galois closure. Greither and Pareigis showed how to find all Hopf--Galois structures on $L/K$. We will call a subextension $L'/K$ of $E/K$ \textit{parallel} to $L/K$ if $[L':K]=n$.

    In this paper, we investigate the relationship between the Hopf--Galois structures on an extension $L/K$ and those on the related parallel extensions. We give an example of a transitive subgroup corresponding to an extension admitting a Hopf--Galois structure but that has a parallel extension admitting no Hopf--Galois structures. We show that once one has such a situation, it can be extended into an infinite family of transitive subgroups admitting this phenomenon. We also investigate this fully in the case of extensions of degree $pq$ with $p,q$ distinct odd primes, and show that there is no example of such an extension admitting the phenomenon. 
\end{abstract}
\maketitle
	
\section{Background and Introduction}\label{intro}
Hopf--Galois structures were discovered by Chase and Sweedler in \cite{CS69} in a theory which describes an analogue of the Galois correspondence for commutative rings, and in particular, for non-Galois field extensions. Using Galois descent, Greither and Pareigis in \cite{GP87} then showed that the problem of finding Hopf--Galois structures on separable (but not necessarily normal) field extensions can be approached entirely group theoretically. Let $L/K$ be a separable extension of degree $n$ with Galois closure $E$. One can form the Galois groups $G:=\Gal(E/K)$ and $G':=\Gal(E/L)$. The result of Greither and Pareigis shows that a $K$-Hopf algebra $H$ giving a Hopf--Galois structure on $L/K$ corresponds to a subgroup $N \leq \text{Perm}(G/G')$ (the abstract isomorphism class of $N$ is known as the `type' of the Hopf--Galois structure) which is regular on the coset space $G/G'$ and normalised by the image, $\lambda(G)$, of the left translation map $\lambda:G \rightarrow \text{Perm}(G/G'), \lambda(g)(hG') \mapsto ghG'$. More precisely, $H=E[N]^G$, the sub-Hopf algebra fixed under the action of $G$. Here, $G$ acts on $E$ by field automorphisms, and on $N$ by conjugation via $\lambda$, that is for $g \in G, \eta \in N$, we have $g \cdot \eta := \lambda(g)\eta\lambda(g)^{-1}$. Greither--Pareigis theory also tells us how $H$ acts on $L/K$, but we do not give that here.

In 1996, Byott in \cite{Byo96} gave a reformulation of this situation by effectively reversing the roles played by $N$ and $G$. Now, given a group $N$ of order, say, $n$, one may construct the \textit{holomorph}, $\Hol(N)$, of $N$. Abusing notation, given the left translation map $\lambda:N \to \mathrm{Perm}(N)$ given by $\lambda(\eta)(\mu):=\eta\mu$ for all $\eta,\mu \in N$, the holomorph of $N$ is defined as the normaliser of $\lambda(N)$ in $\mathrm{Perm}(N)$, and is thus viewed as a permutation group. It can then be shown that
\[\Hol(N) \cong N \rtimes \Aut(N),\]
with multiplication given by
\[[\eta,\alpha][\mu,\beta]:=[\eta\alpha(\mu),\alpha\beta] \;\; \forall \eta,\mu \in N, \forall \alpha,\beta \in \Aut(N).\]
This is the viewpoint we will take for the rest of this paper unless specified otherwise. There is then a natural action of $\Hol(N)$ on $N$ given by
\[[\eta,\alpha]\cdot \mu:=\eta\alpha(\mu).\]
A subgroup $M \leq \Hol(N)$ is then called \textit{transitive} if its action on $N$ is transitive. Given a separable extension $L/K$ with $G,G'$ as above, Byott showed that $N$ is a regular subgroup of $\mathrm{Perm}(G/G')$ normalised by $\lambda(G)$ (with $\lambda$ as initially defined) if and only if there is a transitive subgroup $M \leq \Hol(N)$ and an isomorphism $\phi:G \to M$ with $\phi(G')=M':=\mathrm{Stab}_M(1_N)$. In this setting, we will say that $(G,G')$ is isomorphic to $(M,M')$. Therefore it is possible to hunt for Hopf--Galois structures by looking for transitive subgroups of $\Hol(N)$. That is, given a transitive subgroup $M \leq \Hol(N)$ for some group $N$ of order $n$, any separable extension $L/K$ of degree $n$ such that $(G,G')$ is isomorphic to $(M,M')$ will admit a Hopf--Galois structure of type $N$. Note here that if we start from the holomorph, there is no inverse Galois problem as the base field of the corresponding extension(s) is not assumed. In the case that there are transitive subgroups $M_1 \leq \Hol(N_1)$ and $M_2 \leq \Hol(N_2)$ such that there is an isomorphism $\phi:M_1 \to M_2$ with $\phi(M_1')=M_2'$, then we will say that $(M_1,M_1')$ and $(M_2,M_2')$ are \textit{isomorphic as permutation groups}.

Now let $n$ be a positive integer and suppose we fix a separable field extension $L/K$ of degree $n$ with Galois closure $E$. Then Galois theory tells us that there may be other degree $n$ subextensions $L'/K$ of $E$ (see Figure \ref{parallel}). We will call these related extensions \textit{parallel extensions} with respect to $L/K$. It can then be asked how the extension $L'/K$ behaves relative to $L/K$ in terms of the Hopf--Galois structures that it admits. Indeed, if $L/K$ is Galois, then $E=L=L'$ and so there are no parallel extensions to $L/K$. If $L/K$ is not Galois, then interesting phenomena may occur. For example, let $L/K$ be a degree $4$ extension with $\Gal(E/K) \cong D_8$, the dihedral group of order $8$. Note that $D_8$ has four non-normal subgroups of order $2$ (one of which corresponds to $L/K$), and one normal subgroup $H$ of order $2$. Thus there is a subextension $L'/K$ of $E/K$, of degree $4$ which is Galois (see Figure \ref{nontriv_example}). In this case, the extension $L'/K$ has smaller Galois closure (i.e.~$L'$ itself) than $L/K$, and therefore admits Hopf--Galois structures with smaller permutation group. This example also highlights that the relation of extensions being parallel is not symmetric. In particular, $L'/K$ is parallel to $L/K$, but not vice versa, as $L$ is not contained in the Galois closure of $L'/K$.
\begin{center}
	\begin{figure}[H]
		\begin{tikzcd}
			E \arrow[d,dash] \arrow[dr,dash] & \\
			L \arrow[d,dash,"n"'] & L' \arrow[dl,dash,"n"] \\
			K &  
		\end{tikzcd}
	\caption{A parallel extension, $L'/K$, to $L/K$.}
	\label{parallel}
	\end{figure}	
\end{center}
\begin{center}
	\begin{figure}[H] \label{dihedral}
		\begin{tikzcd}
			E \arrow[d,dash,"2"'] \arrow[dd,dash,"D_8"', bend right=45] \arrow[dr,dash,"2"] & \\
			L \arrow[d,dash,"4"'] & L' \arrow[dl,dash,"4"] \arrow[dl,dash,"D_8/H", bend left = 45] \\
			K &  
		\end{tikzcd}
		\caption{Example showing a Galois parallel extension.}
		\label{nontriv_example}
	\end{figure}	
\end{center}
Using this example as motivation, it is interesting and fruitful to study such phenomena on large classes of extensions where the Hopf--Galois structures have already been determined.

To this end, suppose that we have a classification of Hopf--Galois structures on separable extensions of degree, say, $n$. Further suppose that $L/K$ is one such extension which admits a Hopf--Galois structure of type $N$, for some group $N$ of order $n$. Recall that $G:=\Gal(E/K)$ where $E/K$ is the Galois closure of $L/K$. Thus an index $n$ subgroup $H$ of $G$ corresponds to a degree $n$ sub-extension $L'/K$ of $E/K$, that is, a parallel extension to $L/K$. The Hopf--Galois structures admitted by $L'/K$ then depend on its Galois closure $E'$, which, as highlighted in Figure \ref{nontriv_example}, may be strictly smaller than $E$. This may be computed by first taking the largest normal subgroup $C$ of $G$ contained in $H$, that is  $C=\mathrm{Core}_G(H):=\bigcap_{g \in G}gHg^{-1}$. Then $E'=E^C$ satisfies $\Gal(E'/K)\cong G/C$ and $\Gal(E'/L')\cong H/C$. In other words, asking whether $L'/K$ admits a Hopf--Galois structure of type $M$ is the same as asking whether there is a transitive subgroup $J \leq \Hol(M)$ and an isomorphism $\phi: G/C \to J$ such that $\phi(H/C)=J':=\mathrm{Stab}_J(1_M)$.

In this paper, we start by giving some necessary definitions and present some results relating to parallel extensions to an extension of squarefree degree. In Section \ref{par_pq}, we specialise to the case of extensions of degree $pq$ where $p$ and $q$ are distinct odd primes. In the final section, we give some examples of permutation groups (of non-squarefree degree) which correspond to separable extensions admitting a Hopf--Galois structure, but having a parallel extension admitting no Hopf--Galois structure. We also give a way of constructing an infinite family of examples. 

\section{Preliminaries}\label{par_prelim}
In this section, we give the results needed for the rest of the paper.

Let $n$ be a positive integer and $N$ any group of order $n$. We denote an element of $\Hol(N)$ by $[\eta,\alpha]$ where $\eta \in N$ and $\alpha \in \Aut(N)$. We will write $\eta$ for $[\eta,\text{id}_N]$ and $\alpha$ for $[1_N,\alpha]$ to simplify notation. We also write $\mathbb{Z}_m$ for the ring of integers modulo $m$.
\begin{remark}
    Let $G$ be a transitive subgroup of $\Hol(N)$ with $|N|=n$ and let $L/K$ be a separable extension corresponding to $G$ as in Section \ref{intro}. Let $H_1,H_2$ be two index $n$ subgroups of $G$, let $C_1$ and $C_2$ denote the normal cores of $H_1$ and $H_2$ respectively in $G$, and let $L_1/K$ and $L_2/K$ be the respective corresponding extensions. We note that the relationship between $H_1$ and $H_2$ greatly determines how $L_1/K$ and $L_2/K$ behave with respect to each other from the perspective of Hopf--Galois theory:

    If, for example, $H_1$ and $H_2$ are in the same orbit under $\Aut(G)$, then given a $\phi \in \Aut(G)$ such that $\phi(H_1)=H_2$, it induces an isomorphism $\overline{\phi}:G/C_1 \to G/C_2$ such that $\overline{\phi}(H_1/C_1)=H_2/C_2$. If there is then a transitive subgroup $M \leq \Hol(N')$ for some group $N'$ of order at most $n$ such that $(G/C_1,H_1/C_1)$ is isomorphic to $(M,M')$, we thus also have that $(G/C_2,H_2/C_2)$ is isomorphic to $(M,M')$. We note that this makes sense as we may view, for example, $G/C_1 \overset{\psi}{\cong} \Gal(E^{C_1}/K)$ with $\psi(H_1/C_1)=\Gal(E^{C_1}/L_1)$, where $E^{C_1}$ is the Galois closure of $L_1/K$. We therefore have that $L_1/K$ admits a Hopf--Galois structure of type $N'$ if and only if $L_2/K$ admits a Hopf--Galois structure of type $N'$. We note that the situation in which $H_2$ and $H_2$ are conjugate in $G$ (in which case $L_1/K$ and $L_2/K$ are conjugate extensions) is a special case of this behaviour; in particular, we get that $C_1=C_2$ and not just $C_1 \cong C_2$.

    It is also possible for $H_1$ and $H_2$ to be isomorphic as abstract groups, but not be in the same $\Aut(G)$-orbit. In this case, there are a number of possibilities for $C_1$ and $C_2$ and hence $E^{C_1}/K$ and $E^{C_2}/K$. In particular, we may have that $[E^{C_1}:K] \neq [E^{C_2}:K]$, and in the extreme case as demonstrated by Figure \ref{dihedral}, it could be that $E^{C_1}/K$ is Galois and $E^{C_2}/K$ is not.

    A particularly interesting scenario is where $E^{C_1}=E^{C_2}$ with $L_1/K$ admitting Hopf--Galois structures, but $L_2/K$ not admitting any. With respect to the extension $L/K$ we started off with, say $L_1:=L$ and so $H_1=G'$, this would correspond to $H_2$ having trivial core, possibly even $H_2 \cong G'$, but $L_2/K$ behaving very differently from $L/K$. In particular, $(G,H_2)$ is not isomorphic to any transitive subgroup $(M,M')$ of some $\Hol(N')$.
\end{remark}
In light of this remark, we note it is therefore of interest to sort the index $n$ subgroups into conjugacy classes, orbits under $\Aut(G)$, and abstract isomorphism classes.

Let $\pi$ be any set of primes. We recall that a $\pi$-subgroup of a group $G$ is a subgroup $H \leq G$ such that the prime divisors of $|H|$ all lie in $\pi$. Further, a \textit{Hall $\pi$-subgroup} of $G$ is a $\pi$-subgroup $H$ whose order is coprime with $|G:H|$. We also recall Hall's Theorem below, which will be referred to in the rest of this paper.
\begin{theorem}[Hall, 1928. See 9.1.7 of \cite{Rob96}]
    Let $G$ be a finite solvable group and $\pi$ any set of primes dividing $|G|$. Then every $\pi$-subgroup of $G$ is contained in a Hall $\pi$-subgroup of $G$. Further, any two Hall $\pi$-subgroups of $G$ are conjugate.
\end{theorem}
In light of this, for a positive integer $n$, we define
\begin{align*}
    &\pi(n):=\{\text{primes dividing }n\},\\
    &\pi'(n):=\{\text{primes not dividing }n\}.
\end{align*}
We also denote by $\#\mathrm{Cl}_n(G)$ the number of conjugacy classes that the index $n$ subgroups of $G$ fall into.

For the remainder of this section, we let $n$ be any squarefree integer unless stated otherwise, and recall that $N$ is a group of order $n$.
\begin{remark}
    If $N$ is a group of squarefree order, then by Lemma 3.2 of \cite{AB18}, we have that $\Hol(N)$ is solvable. So Hall's theorem applies to any subgroup of $\Hol(N)$.
\end{remark}
We now begin by stating and proving some results for transitive subgroups of $\Hol(N)$.
\begin{lemma}\label{unique_Hall}
    There is a unique Hall $\pi(n)$-subgroup, $Q$, of $\Hol(N)$. Furthermore, let $G$ be any subgroup of $\Hol(N)$. Then $G$ may be written as $G=U \rtimes V$, where $U=G \cap Q$ is the unique Hall $\pi(n)$-subgroup of $G$, and $V$ is a Hall $\pi'(n)$-subgroup of $G$.
\end{lemma}
\begin{proof}
    If $N$ is a group of squarefree order $n$, then (see \cite{AB18}) for some $d,e$ with $n=ed$ and some $k$ with order $d$ modulo $e$, we have
    \[N\cong\langle \sigma,\tau \mid \sigma^e=\tau^d=1,\tau\sigma=\sigma^k\tau \rangle.\]
    Further, let $z:=\gcd(k-1,e)$ and consider the elements $\theta,\{\phi_s\}_{s \in \mathbb{Z}_e^{\times}}$ of $\Aut(N)$ given by
    \begin{align*}
        &\theta(\sigma)=\sigma,		&&\theta(\tau)=\sigma^z\tau,\\
        &\phi_s(\sigma)=\sigma^s,	&&\phi_s(\tau)=\tau.
    \end{align*}
    Then Lemma 4.1 of \cite{AB20} tells us that
    \[\Aut(N)=\langle\theta\rangle \rtimes \{\phi_s\}_{s \in \mathbb{Z}_e^{\times}}.\]
    Noting that $\Phi:=\{\phi_s\}_{s \in \mathbb{Z}_e^{\times}}$ is an abelian group, we further decompose $\Phi$ as $\Phi=\Phi_1\times \Phi_2$ such that $\Phi_1$ is the Hall $\pi(n)$-subgroup of $\Phi$ and $\Phi_2$ is the Hall $\pi'(n)$ subgroup of $\Phi$.

    Now, the subgroup $Q:=N \rtimes (\langle \theta \rangle \rtimes \Phi_1)$ is a normal (and therefore unique) Hall $\pi(n)$-subgroup of $\Hol(N)$, and has $\Phi_2$ as a complementary subgroup. Therefore, we may write $\Hol(N)=Q \rtimes \Phi_2$.

    Let $G$ be any subgroup of $\Hol(N)$. Denote by $U$ the unique maximal normal $\pi(n)$-subgroup of $G$. Clearly we have that $U \leq Q$ as any $\pi(n)$-subgroup of $G \leq \Hol(N)$ must be contained in $Q$. We also have that $G \cap Q \leq U$. This is because $G \cap Q$ is normal in $G$ and $U$ contains every normal $\pi(n)$-subgroup of $G$. Therefore we have that $U=G \cap Q$. It is clear then that $U$ has order coprime to its index in $G$, and therefore by the Schur-Zassenhaus Theorem, $G=U \rtimes V$, where $V$ is some $\pi'(n)$-subgroup of $G$. The Schur-Zassenhaus Theorem also tells us that $V$ is contained in a conjugate subgroup of $\Phi_2$ in $\Hol(N)$.
\end{proof}
Whenever we write $G=U \rtimes V$, we will now assume that $U=G \cap Q$ and $V$ is a Hall $\pi'(n)$-subgroup of $G$ which is a subgroup of a conjugate of $\Phi_2$.
\begin{corollary}\label{index_n_subgroups}
   Let $G=U \rtimes V \leq \Hol(N)$ be any subgroup. Then any index $n$ subgroup $H$ of $G$ can be written as $H=U' \rtimes V'$, where $U'$ is some index $n$ subgroup of $U$ and $V'$ is a subgroup conjugate to $V$ in $G$.
\end{corollary}
\begin{proof}
    This is the same proof as that given for the form of $G$. However, we note that $|V'|=|V|$, and so $V$ and $V'$ must be conjugate in $G$.
\end{proof}
\begin{proposition}\label{conj_orbs}
   Let $G \leq \Hol(N)$ be any subgroup. Let
   \[H_1:=U_1 \rtimes V_1, \text{ and } H_2:=U_2 \rtimes V_2\]
   be index $n$ subgroups of $G$ as described in Corollary \ref{index_n_subgroups}. Then
   \begin{enumerate}[label=(\roman*)]
       \item The subgroups $H_1$ and $H_2$ are conjugate in $G$ if and only if $U_1$ and $U_2$ are conjugate in $G$.
       \item The subgroups $H_1$ and $H_2$ are in the same $\Aut(G)$-orbit if and only if $U_1$ and $U_2$ are in the same $\Aut(G)$-orbit.
       \item The subgroups $H_1$ and $H_2$ are isomorphic as abstract groups if and only if $U_1$ and $U_2$ are isomorphic as abstract groups. 
   \end{enumerate}
\end{proposition}
\begin{proof}
\textit{(i)} Suppose first that there is a $g \in G$ such that $gH_1g^{-1}=H_2$. Then
\[gU_1g^{-1} \rtimes gV_1g^{-1}=g(U_1 \rtimes V_1)g^{-1}=U_2 \rtimes V_2.\]
If $gU_1g^{-1} \not\subseteq U_2$, then there is a $u \in U_1$ such that $gug^{-1} \in H_2\setminus U_2$. However, this contradicts the fact that $U_2$ is the unique Hall $\pi(n)$-subgroup of $H_2$. This is because $u$ and hence also $gug^{-1}$ have order a product of primes in $\pi(n)$. So we must have $gU_1g^{-1} = U_2$.

Now suppose that there is a $u \in G$ such that $uU_1u^{-1}=U_2$. We show that there is a $g \in G$ such that $g(U_1 \rtimes V_1)g^{-1} = U_2 \rtimes V_2$. We claim that there is a $u' \in \mathrm{Norm}_G(U_2)$ such that $u'V_2u'^{-1}=uV_1u^{-1}$. Given this, we see that
\[u'(U_2 \rtimes V_2)u'^{-1}=u'U_2u'^{-1} \rtimes u'V_2u'^{-1}=U_2 \rtimes uV_1u^{-1}=u(U_1 \rtimes V_1)u^{-1}.\]
Therefore $H_1$ and $H_2$ are conjugate in $G$ (that is $g:=u'^{-1}u)$. To show the existence of $u'$, we note that both $V_2$ and $uV_1u^{-1}$ are subgroups of $\mathrm{Norm}_G(U_2)$. This is clear for $V_2$, but for $uV_1u^{-1}$, we have $V_1 \leq \mathrm{Norm}_G(U_1)$, and so
\[uV_1u^{-1} \leq u\mathrm{Norm}_G(U_1)u^{-1}=\mathrm{Norm}_G(uU_1u^{-1})=\mathrm{Norm}_G(U_2).\]
Finally, we note that both $V_2$ and $uV_1u^{-1}$ are Hall $\pi'(n)$-subgroups of $\mathrm{Norm}_G(U_2)$, and so they must be conjugate by an element, say $u' \in \mathrm{Norm}_G(U_2)$.

\textit{(ii)} We now show $H_1$ and $H_2$ are in the same $\Aut(G)$-orbit if and only if $U_1$ and $U_2$ are in the same $\Aut(G)$-orbit. Indeed, if there is a $\phi \in \Aut(G)$ such that $\phi(H_1)=H_2$, then $\phi(U_1 \rtimes V_1)=U_2 \rtimes V_2$. We must have $\phi(U_1) = U_2$ as $U_1,U_2$ are the unique Hall $\pi(n)$-subgroups in $H_1,H_2$ respectively. Therefore, now suppose that there is a $\phi \in \Aut(G)$ such that $\phi(U_1)=U_2$. Then $\phi(U_1 \rtimes V_1)=\phi(U_1) \rtimes \phi(V_1)=U_2 \rtimes \phi(V_1)$. Now clearly $U_2$ is conjugate to $U_2$, so by the first part of this proof, $U_2 \rtimes \phi(V_1)$ must be conjugate to $U_2 \rtimes V_2$. Composing $\phi$ with this conjugation, we see that $H_1$ and $H_2$ are in the same $\Aut(G)$-orbit.

\textit{(iii)} Finally, we note that the argument for $(ii)$ can be replicated for $\phi:H_1 \to H_2$ now an isomorphism of abstract groups.
\end{proof}
\begin{corollary}
    Suppose $G=U \rtimes V$ is a transitive subgroup of $\Hol(N)$. Then for any index $n$-subgroup $H$ of $G$, $\mathrm{Core}_G(H)=\mathrm{Core}_G(H\cap U)$.
\end{corollary}
\begin{proof}
    We have $\mathrm{Core}_G(H)\supseteq\mathrm{Core}_G(H\cap U)$. As mentioned in Lemma  \ref{unique_Hall}, we have that $G=U \rtimes V$, where $U:=G \cap Q$ and $V$ is some subgroup of $G$ of order coprime to $n$. By Corollary \ref{index_n_subgroups}, we also have that any index $n$ subgroup $H$ of $G$ has the form $H=U_1 \rtimes V_1$, where $U_1$ is an index $n$ subgroup of $U$ and $V_1$ is conjugate to $V$ in $G$.

    Suppose then that $C:=\mathrm{Core}_G(H)=U_2 \rtimes V_2$ where $U_2 \leq U_1$ and $V_2$ is a complement of $U_2$ in $C$ (that is $V_2$ is conjugate to some subgroup of $V_1$ by Schur-Zassenhaus). In particular, $C$ is normal in $G$, and so $gV_2g^{-1} \leq C$ for all $g \in G$. Therefore, for each $v_2 \in V_2$, there exist a $u' \in U_2$ and $v' \in V_2$ such that $gv_2g^{-1}=u'v'$. However, we then have $u'=gv_2g^{-1}v'^{-1}$, and as $u'$ has coprime order with both $v_2$ and $v'$, we must have that $u'=1$, and so $V_2$ must be normal in $G$. In particular, we have that $V_2=\mathrm{Core}_G(V_2)$. We claim that $V_2$ is trivial and so $C=U_2=\mathrm{Core}_G(H \cap U)$.
    
    Since $G$ is a transitive subgroup of $\Hol(N)$, the index $n$ subgroup $G'=\mathrm{Stab}_G(1_N)$ of $G$ has trivial core. Now by Corollary \ref{index_n_subgroups}, we have that $G'=U_3 \rtimes V_3$ where $V_3$ is a $\pi'(n)$-subgroup of $G$. In particular, $V_1$ is conjugate to $V_3$ in $G$ and so

    \[\mathrm{Core}_G(V_2) \subseteq \bigcap_{g \in G}gV_1g^{-1}=\bigcap_{g \in G}gV_3g^{-1} \subseteq \bigcap_{g \in G}gG'g^{-1}=\{1\}.\]
\end{proof}
\section{Parallel extensions of degree \texorpdfstring{$pq$}{pq}}\label{par_pq}
We now specialise to the case $n=pq$ with $p>q$ distinct odd primes. In this section, we investigate the Hopf--Galois structures on parallel extensions of degree $n$. We start with a few more results.
\begin{lemma}\label{pq_Burnside}
    Let $L/K$ be a separable extension of degree $pq$ with $p \not\equiv 1\pmod{q}$. Then any parallel extension to $L/K$ admits a unique Hopf--Galois structure, which is of cyclic type.
\end{lemma}
\begin{proof}
    If $p \not\equiv 1\pmod{q}$, then $\gcd(pq,\varphi(pq))=1$, so $pq$ is a Burnside number. Theorem 1 of \cite{Byo96} then tells us that every separable extension of degree $pq$ admits a unique Hopf--Galois structure, which is of cyclic type. If $L/K$ is a separable extension of degree $pq$, then any parallel extension to $L/K$ also has degree $pq$ and thus also admits a unique Hopf--Galois structure, which is also of cyclic type.
\end{proof}
\begin{lemma}\label{triv_core}
    Let $L_1/K, L_2/K$ be two separable extensions of degree $pq$ with the same Galois closure $E$ such that they both admit Hopf--Galois structures. Then $L_1/K$ admits a Hopf--Galois structure of type $N$ if and only if $L_2/K$ admits a Hopf--Galois structure of type $N$.
\end{lemma}
\begin{proof}
    Let $G:=\Gal(E/K)$, $G_1:=\Gal(E/L_1)$ and $G_2:=\Gal(E/L_2)$. If both $L_1/K$ and $L_2/K$ admit Hopf--Galois structures, then there are transitive subgroups $M_1 \leq \Hol(N_1)$, $M_2 \leq \Hol(N_2)$ with $M_1':=\mathrm{Stab}_{M_1}(1_{N_1})$, $M_2':=\mathrm{Stab}_{M_2}(1_{N_2})$ and isomorphisms $\phi_1:G \to M_1$, $\phi_2:G \to M_2$ with $\phi_1(G_1)=M_1'$ and $\phi_2(G_2)=M_2'$ respectively.
    
    It was observed in section 3 of \cite{Dar24(a)} (immediately after the proof of Proposition 3.19) that $M_1$ and $M_2$ are isomorphic as permutation groups if and only if they are isomorphic as abstract groups. Given that $M_1 \cong G \cong M_2$, we therefore must have that $M_1$ and $M_2$ are also isomorphic as permutation groups, and therefore both $L_1/K$ and $L_2/K$ admit Hopf--Galois structures of types $N_1$ and $N_2$.
\end{proof}
\begin{remark}
    Let $G \leq \Hol(N)$ be a transitive subgroup for some group $N$ of order $pq$, and suppose $G$ has an index $pq$ subgroup $H$ such that $C:=\mathrm{Core}_G(H)=\{1\}$. By the discussion in Section \ref{intro}, this situation corresponds to a separable extension $L/K$, with Galois closure $E/K$, admitting a Hopf--Galois structure of type $N$, and a parallel extension $L'/K$ with the same Galois closure.

    If $L'/K$ admits a Hopf--Galois structure, then Lemma \ref{triv_core} tells us that the Hopf--Galois structures on $L'/K$ will be of the same types as those on $L/K$. In fact by the proof of Lemma \ref{triv_core}, in order to check whether $L'/K$ admits Hopf--Galois structures, we need only check whether there is an isomorphism $\phi:G \cong G/C \to G$ such that $\phi(H/C)=\mathrm{Stab}_G(1_N)$. This allows us to quickly test $L'/K$ for Hopf--Galois structures.
\end{remark}
 Given Lemma \ref{pq_Burnside}, we will assume that $p \equiv 1\pmod{q}$ for the rest of this section. There are then two groups $N$ of order $pq$, that is the cyclic group $C_{pq}$ and the non-abelian metacyclic group $C_p \rtimes C_q$. In each case, we describe $N$ using two generators $\sigma, \tau$ of orders $p$ and $q$ respectively. For $C_{pq}$, we impose the relation $\sigma\tau=\tau\sigma$, and for $C_p\rtimes C_q$, we set $\tau\sigma=\sigma^k\tau$ where $k$ has order $q \bmod{p}$. Sections \ref{cyclic} and \ref{metacyclic} will consider these two groups in turn by working with subgroups of $\Hol(N)$ which are transitive on $N$.
 
\subsection{Transitive subgroups of \texorpdfstring{$\Hol(N)$}{Hol(N)}}\label{trans_subs}
We now recall the list of transitive subgroups of $\Hol(N)$ respectively for $N$ cyclic (I) and $N$ non-abelian (II) of order $pq$. These are computed in \cite{Dar24(a)}. Let $e_0$ be the largest power of $q$ dividing $p-1$, and $s=(p-1)/q^{e_0}$.

(I) Let $N\cong C_{pq}$, and let $\alpha$ generate the unique (cyclic) Sylow $q$-subgroup of $\Aut(\langle\sigma\rangle)$, which has order $q^{e_0}$. Given some $1 \leq c \leq e_0$ and $1 \leq t \leq q^c-1$ with $q \nmid t$, we define
\[J_{t,c}:=\left\langle \sigma, \left[\tau,\alpha^{tq^{e_0-c}}\right]\right\rangle.\]
Then the subgroups of $\Hol(N)$ transitive on $N$ are given in Table \ref{cyclic-trans-subgroups}.
\begin{table}[H]
    \setlength{\extrarowheight}{3.5mm}
    \text{ .}
    \centering
    \scalebox{1.1}{
    \begin{tabular}{|c|c|}
        \hline
	Parameters	&	Transitive subgroup\\
	\hline
	$X \leq \Aut(N)$ &   $N \rtimes X$\\[5pt]
	\hline
	$Y \leq \Aut(\langle\sigma\rangle)$ s.t. $\alpha \not\in Y$,  &	$J_{t,c} \rtimes Y$\\
        $1 \leq c \leq e_0, 1 \leq t \leq q^c-1$ s.t. $q \nmid t$   &\\[5pt]
	\hline
    \end{tabular}}
    \vskip2mm
    \caption{Transitive subgroups of $\Hol(N)$ for $N$ cyclic}
    \label{cyclic-trans-subgroups}
\end{table}
(II) Let $N \cong C_p \rtimes C_q$. Let $\theta \in \Aut(N)$ be such that $\theta(\sigma)=\sigma$ and $\theta(\tau)=\sigma\tau$, hence $\theta$ has order $p$. One may view $\Hol(N)$ as $\mathbb{F}_p^2\rtimes \langle T,A,B \rangle$, where $\e_1:=(1,0)^{\textbf{t}}$ and $\e_2:=(0,1)^{\textbf{t}}$ are the basis vectors of $P:=\mathbb{F}_p^2$ corresponding to $\sigma$ and $\sigma\theta^{k-1}$ respectively, and $T,A,B$ are commuting linear maps corresponding to $\tau$ (which has order $q$) and elements of $\Aut(N)$ of orders $q^{e_0}$ and $s$ respectively. These act on the unique Sylow $p$-subgroup $P$ of $\Hol(N)$ as follows:
\begin{align*}
    &T\e_1=a_{\alpha}^{q^{e_0-1}}\e_1=k\e_1,	&&T\e_2=\e_2,\\
    &A\e_1=a_{\alpha}\e_1,				&&A\e_2=a_{\alpha}\e_2,\\
    &B\e_1=a_{\beta}\e_1,				&&B\e_2=a_{\beta}\e_2.
\end{align*}
Here $a_{\alpha}$ and $a_{\beta}$ have orders $q^{e_0}$ and $s$ modulo $p$ respectively. Table \ref{metacyclic-trans-subs} gives a list of the transitive subgroups of $\Hol(N)$ whose order is divisible by $p^2$. By Theorem 3.22 of \cite{Dar24(a)}, every transitive subgroup of $\Hol(N)$ whose order is not divisible by $p^2$ is isomorphic, as a permutation group, to a transitive subgroup of $\Hol(C_{pq})$. As our discussion in this paper depends only on the isomorphism classes of the permutation groups, we need not revisit these groups here.
\begin{table}[H]
    \setlength{\extrarowheight}{3.5mm}
    \text{ .}
    \centering
    \scalebox{1.1}{
    \begin{tabular}{|c|c|c|}
        \hline
	Parameters	&	Group	&	Order\\
	\hline
	$0\leq c \leq e_0$,  $d|s$ &   $P \rtimes \left\langle T,A^{q^{e_0-c}},B^{s/d} \right\rangle$	&	$p^2q^{1+c}d$\\[5pt]
	\hline
	$1\leq c\leq e_0$,  $d|s$, $u \in \mathbb{Z}_{q^c}^{\times}$&	$P \rtimes \left\langle TA^{uq^{e_0-c}},B^{s/d} \right\rangle$	&	$p^2q^cd$\\[5pt]
	\hline
    \end{tabular}}
    \vskip2mm
    \caption{Transitive subgroups of $\Hol(N)$ of order divisible by $p^2$ for $N$ non-abelian}
    \label{metacyclic-trans-subs}
\end{table}
We start by investigating the parallel extensions to those admitting Hopf--Galois structures of cyclic type.

\subsection{Cyclic case}\label{cyclic}
Let $N$ be the cyclic group of order $pq$, let $G$ be a transitive subgroup of $\text{Hol}(N)$ of order $|G|=r$, and let $Q:=\langle \sigma,\tau,\alpha \rangle$, the unique Hall $\left\{p,q\right\}$-subgroup of $\Hol(N)$ of order $pq^{1+e_0}$.
\begin{proposition}\label{cyclic-conjugacy}
    Let $X,Y,t$ and $c$ be as in section \ref{trans_subs}.
	
    If $G=N \rtimes X$, then
    \[\#\mathrm{Cl}_{pq}(G)=\begin{cases*}
    1, \;\; \text{if }q^2 \nmid r,\\
    2, \;\; \text{if } q^2 \mid r \text{ and } \tau \notin Z(G),\\
    q+1, \;\; \text{if } q^2 \mid r \text{ and } \tau \in Z(G).
    \end{cases*}\]

    If $G=J_{t,c} \rtimes Y$, then
    \[\#\mathrm{Cl}_{pq}(G)=1.\]
Further, for any such $G$, all parallel extensions to any extension associated with $G$ admit at least one Hopf--Galois structure of cyclic type.
\end{proposition}
\begin{proof}
    We begin by noting that for any transitive subgroup $G<\Hol(N)$ such that $q^2 \nmid r$, any index $pq$ subgroup of $G$ is a Hall $\pi'(pq)$-subgroup of $G$, and so they are all conjugate. Thus we assume that $q^2 \mid r$ for the remainder of the proof, and thus we let $c+1$ be the largest power of $q$ dividing $r$ (and so $1 \leq c \leq e_0$). For the index $pq$ subgroups $H$ of $G$, we also note that by Proposition \ref{conj_orbs}, we need only look at the conjugacy classes of $H \cap Q$ in $G$. 
	
    Now fix a subgroup $X \leq \Aut(N)$ and consider the transitive subgroup $G=N \rtimes X$ of $\Hol(N)$. Let $H$ be an index $pq$ subgroup of $G$. Then it can be shown that $H \cap Q$ has one of the following forms:
    \[H_1=\left\langle \left[\sigma^a\tau^b,\alpha^{q^{e_0-c}}\right]\right\rangle \;\; \text{ or } \;\; H_2=\left\langle \left[\sigma^a,\alpha^{q^{e_0-(c-1)}}\right],\tau \right\rangle,\]
    for some $0 \leq a \leq p-1$ and $0 \leq b \leq q-1$, noting that, for subgroups of type $H_2$, we have $a=0$ whenever $c=1$. With a view to computing $\#\mathrm{Cl}_{pq}(G)$, we first note that we can conjugate any of these subgroups by a suitable power of $\sigma$ to obtain a subgroup of the same type but with $a=0$. Given this observation, we abuse notation and now work with the representatives
    \begin{equation}\label{index_gen}
        H_1=\left\langle \left[\tau^b,\alpha^{q^{e_0-c}}\right]\right\rangle \;\; \text{ or } \;\; H_2=\left\langle \alpha^{q^{e_0-(c-1)}},\tau \right\rangle.
    \end{equation}
    If $\tau \notin Z(G)$, then there is an element $\beta \in G$ such that conjugation by $\beta$ gives a non-trivial automorphism of $\langle\tau\rangle$. Thus we have, say, $\beta\tau=\tau^x\beta$ for some $1 \leq x \leq q-1$. We claim that we can assume $\beta \in H$. This is because $\beta$ has order coprime with $pq$ and so lies in a Hall $\pi'(pq)$-subgroup of $G$. Further by Corollary \ref{index_n_subgroups}, we have that $H$ contains another Hall $\pi'(pq)$-subgroup of $G$ and so also an element $\beta'$ conjugate to $\beta$. Finally, as both $N$ and $\Aut(N)$ are abelian, and so it can be shown that $\beta'\tau\beta'^{-1}=\beta\tau\beta^{-1}$.
    
    Now, if we conjugate the generator $\left[\tau^b,\alpha^{q^{e_0-c}}\right]$ of $H_1$ by $\beta$ and $\beta^2$ respectively and assume that $b \neq 0$, we see that $\left[\tau^{bx},\alpha^{q^{e_0-c}}\right],\left[\tau^{2bx},\alpha^{q^{e_0-c}}\right] \in H_1$. In particular, we then get that $\tau^{bx} \in H_1$, and thus $\tau,\alpha^{q^{e_0-c}} \in H_1$. However, this then means that $q^{c+1} \mid |H_1|$, which gives a contradiction as then $H_1$ has index $p$ in $G$. So we must have $b=0$ in (\ref{index_gen}). We note that $H_1$ is cyclic and $H_2$ is not, unless $c=1$. Thus, for $c>1$, we obtain two conjugacy classes of index $pq$ subgroups of $G$, represented by $H_1$ and $H_2$. Each of these classes is of size $p$. For $c=1$, we get the families represented by
    \begin{equation}
        H_1=\left\langle \alpha^{q^{e_0-1}}\right\rangle \;\; \text{ or } \;\; H_2=\left\langle \tau \right\rangle.
    \end{equation}
    It is clear that they remain in distinct conjugacy classes as $\langle \tau \rangle$ is a normal subgroup in $G$. Thus $H_1$ represents a conjugacy class of size $p$, and $H_2$ represents a class of size $1$.

    We now compute $C:=\mathrm{Core}_G(H)$ for each index $pq$ subgroup $H$ of $G$, with a view to seeing whether $G/C$ appears as a transitive subgroup $J$ of $\Hol(M)$ with $\mathrm{Stab}_J(1_M) \cong H/C$ for $M=C_{pq}$ or $M=C_p \rtimes C_q$. We will find that taking $M=C_{pq}=N$ is sufficient, showing that all parallel extensions in this situation admit a Hopf--Galois structure of cyclic type.
	
    For an index $pq$ subgroup $H$ containing a group of type $H_1$, we see that $H$ is conjugate to $G'=\mathrm{Stab}_G(1_N)$, and therefore $C=\{1\}$. The corresponding parallel extension is therefore conjugate to the original extension and so admits a Hopf--Galois structure of cyclic type.
	
    For an index $pq$ subgroup $H$ containing a group of type $H_2$, we compute that $C=\langle\tau\rangle$. For some $X' \leq \Aut(N)$ of coprime order to $pq$, we therefore have:
    \[G/C \cong \langle \sigma,\alpha^{q^{e_0-c}} \rangle \rtimes X'\]
    and
    \[H/C \cong \left\langle\left[\sigma^a,\alpha^{q^{e_0-(c-1)}}\right]\right\rangle \rtimes X'.\]
    Consider the isomorphism $\phi: G/C \to J_{1,c} \rtimes X'$ given by:
    \begin{align*}
        &\phi(\sigma)=\sigma,	&&\phi(\alpha^{q^{e_0-c}})=\left[\sigma^{-a}\tau,\alpha^{q^{e_0-c}}\right],
    \end{align*}
    with $\phi$ acting as identity on $X'$. Then we note that $\phi$ sends $H/C$ to
    \[\left\langle \alpha^{q^{e_0-(c-1)}} \right\rangle \rtimes X'=\mathrm{Stab}_{J_{1,c}\rtimes X'}(1_{C_{pq}})\]
    giving an isomorphism from $(G/C,H/C)$ to a transitive subgroup of $\Hol(C_{pq})$. The corresponding parallel extension therefore admits the same Hopf--Galois structures as an extension with Galois closure isomorphic to $J_{1,c} \rtimes X'$.
		
    Suppose now that $\tau \in Z(G)$. We obtain no restrictions on $b$. In this case, conjugating by any element does not change the power of $\tau$, and thus $H_1$ represents $q$ conjugacy classes, each of size $p$. By a similar argument as above, $H_2$ again represents a single conjugacy class (of size $p$ if $c>1$ and of size $1$ if $c=1$) which is distinct from any class represented by $H_1$.
	
    The core of any conjugacy class represented by $H_1$ is trivial. For some $X' \leq \Aut(N)$ of coprime order to $pq$, we therefore have:
    \[G/C=G/\{1\} \cong G = \left\langle \sigma,\tau,\alpha^{q^{e_0-c}} \right\rangle \rtimes X'\]
    and
    \[H/C=H/\{1\} \cong H = \left\langle \left[\tau^b,\alpha^{q^{e_0-c}}\right] \right\rangle\]
    for some $0 \leq b \leq q-1$. Consider now the isomorphism $\phi: G \to G$ given by
    \[\phi\left(\alpha^{q^{e_0-c}}\right)=\left[\tau^{-b},\alpha^{q^{e_0-c}}\right]\]
    and acting as identity on $N$ and $X'$. Then we note that $\phi$ sends $H$ to \[G'=\left\langle \alpha^{q^{e_0-c}} \right\rangle \rtimes X',\]
    giving an isomorphism from $(G/\{1\},H/\{1\})$ to a transitive subgroup of $\Hol(C_{pq})$. The corresponding parallel extension therefore admits a Hopf--Galois structure of cyclic type.
    
    Any index $pq$ subgroup containing a group of $H_2$ again has core $\langle\tau\rangle$. Similar arguments as above show that each parallel extension admits a Hopf--Galois structure of cyclic type.
	
    Finally, for some fixed $t,c$, we consider the case $G=J_{t,c} \rtimes Y$. In this case, $G \cap Q = J_{t,c}$. Let $H$ be an index $pq$ subgroup of $G$. Then $H \cap Q$ is of the form
    \[\left\langle \left[\sigma^a,\alpha^{q^{e_0-(c-1)}t}\right]\right\rangle,\]
    for some $0 \leq a \leq p-1$. We see that any group of this form is conjugate, by some $\sigma^{a'} \in G$, to any other, and so all such $H$ are in the same conjugacy class. As $G'=\mathrm{Stab}_{G}(1_N)$ is one such index $pq$ subgroup of $G$ and is in the same conjugacy class as any other $H$, all corresponding parallel extensions must therefore admit a Hopf--Galois structure of cyclic type.
\end{proof}
We end this section with the classification of the index $pq$ subgroups into both $\Aut(G)$ orbits and abstract isomorphism classes.
\begin{proposition}
    Let $G$ be a transitive subgroup of $\Hol(N)$. If $\#\mathrm{Cl}_{pq}(G)>1$, the index $pq$ subgroups of $G$ form two distinct orbits under $\Aut(G)$. If $q^3 \mid r$, they form two distinct (abstract) isomorphism classes. If $q^3 \nmid r$, they form a single (abstract) isomorphism class.
\end{proposition}
\begin{proof}
    Let $H$ be an index $pq$ subgroup of $G=N\rtimes X$ for some $X \leq \Aut(N)$ and let $c+1$ be the largest power of $q$ dividing $r$ (as $\#\mathrm{Cl}_{pq}(G)>1$, we must have $q^2 \mid r$ by Proposition \ref{cyclic-conjugacy} and therefore $c \geq 1$). By Proposition \refeq{conj_orbs}, we need only consider the $\Aut(G)$-orbit of $H \cap Q$. We will work with representatives from the conjugacy classes of $H \cap Q$ worked out in Proposition \ref{cyclic-conjugacy}:
    \[H_{1,b}=\left\langle\left[\tau^b,\alpha^{q^{e_0-c}}\right]\right\rangle \;\; \text{ and } \;\; H_2=\left\langle \alpha^{q^{e_0-(c-1)}},\tau \right\rangle,\]
    where $b=0$ in the case that $\tau \notin Z(G)$, and $0 \leq b \leq p-1$ otherwise.
	
    The number of isomorphism classes (as abstract groups) may be worked out as follows:
    
    If $q^3 \nmid r$, then $c=1$ and each representative is cyclic of order $q$. In particular, we have:
    \begin{equation}\label{c_eq_1}
        H_{1,b}=\left\langle \left[\tau^b,\alpha^{q^{e_0-1}}\right] \right\rangle \cong C_q \cong \langle \tau \rangle = H_2.
    \end{equation}    
    If $q^3 \mid r$, then $c>1$, and so, whereas $H_{1,b}$ is cyclic for each $b$, we have that
    \[H_2 \cong C_q \times C_{q^{c-1}}\]
    is not. In the latter case, we therefore get two (abstract) isomorphism classes for the subgroups $H \cap Q$. Proposition \ref{conj_orbs} tells us that this result may then be lifted to the index $pq$-subgroups $H$ of $G$.
	
    We now consider the $\Aut(G)$-orbits of these representatives.    
    
    Let $\phi \in \Aut(G)$. Then by an order argument, we must have
    \begin{align*}
        &\phi(\sigma)=\sigma^{a_1},\\			&\phi(\tau)=\left[\sigma^{a_2}\tau^{b_2},\alpha^{c_2q^{e_0-1}}\right],\\
        &\phi\left(\alpha^{q^{e_0-c}}\right)=\left[\sigma^{a_3}\tau^{b_3},\alpha^{c_3q^{e_0-c}}\right]
    \end{align*}
    for some $1 \leq a_1 \leq p-1$, $0 \leq a_2,a_3 \leq p-1$, $1 \leq b_2 \leq q-1$, and $0 \leq b_3,c_2,c_3 \leq q-1$. Using the relations
    \[\sigma\tau=\tau\sigma, \; \alpha(\sigma)=\sigma^{a_{\alpha}}, \; \alpha(\tau)=\tau,\]
    we get that $a_2=c_2=0$ and $c_3=1$. In particular, we have:
    \begin{align*}
        &\phi(\sigma)=\sigma^{a_1},\\
        &\phi(\tau)=\tau^{b_2},\\
        &\phi\left(\alpha^{q^{e_0-c}}\right)=\left[\sigma^{a_3}\tau^{b_3},\alpha^{q^{e_0-c}}\right].
    \end{align*}
    Thus the subgroup $\langle \tau \rangle$ is characteristic in $G$, and so no $H_{1,b}$ can be in the same $\Aut(G)$-orbit as $H_2$ as $\tau \notin H_{1,b}$.
    
    Now suppose that $\beta \in X$, then
    \[\phi(\beta)=[\sigma^i\tau^j,\beta']\]
    for some $0 \leq i \leq p-1, 0 \leq j \leq q-1$ and some $\beta' \in X$ of the same order as $\beta$. Given that $\alpha$ and $\beta$ commute, and suppose that $\beta'(\sigma)=\sigma^x$ for some $1\leq x \leq p-1$, $\beta'(\tau)=\tau^y$ for some $1 \leq y \leq q-1$, we therefore have that
    \[\left[\sigma^{a_3+ia_{\alpha}^{q^{e_0-c}}}\tau^{j+b_3},\alpha^{q^{e_0-c}}\beta'\right]=\left[\sigma^{i+xa_3}\tau^{j+yb_3},\alpha^{q^{e_0-c}}\beta' \right].\]
    In particular, we have that
    \[a_3(1-x)=i\left(1-a_{\alpha}^{q^{e_0-c}}\right), \; \text{and} \; b_3(1-y)=0.\]
    We now have
    \[\phi\left(\left[\tau^b,\alpha^{q^{e_0-c}}\right]\right)=\left[\sigma^{a_3}\tau^{bb_2+b_3},\alpha^{q^{e_0-c}}\right].\]
    We note that if $\tau \in Z(G)$, then we must have $\beta'(\tau)=\tau$ and so $y=1$, meaning there are no restrictions on $b_3$. In this case, we note that by choosing a suitable $b_3$ and by composing $\phi$ with conjugation by $\sigma^a$ for a suitable $a$, we see that for any $0 \leq b,b' \leq q-1$, $H_{1,b}$ and $H_{1,b'}$ are both in the same $\Aut(G)$-orbit.

    If $\tau \notin Z(G)$, then we must have $b=0$ and so the groups of type $H_{1,b}=H_{1,0}$ form a single $\Aut(G)$-orbit.
\end{proof}
For the transitive subgroups $G \leq \Hol(N)$ described in Table \ref{cyclic-trans-subgroups}, the above data is summarised in Table \ref{cyclic-index-pq}. 
\begin{table}[H]
	\centering
	\scalebox{1.1}{
	\begin{tabular}{|c|c|c|c|}
		\hline
		Transitive subgroup	$G$ &	$\#\mathrm{Cl}_{pq}(G)$	&	$\#\text{Aut}(G)$-orbits	&	$\#$Isom. classes\\
		\hline
		$N \rtimes X$, $q^2\nmid r$&	$1$	&	$1$ 						& 	$1$ \\
		\hline
		$N \rtimes X$, $q^2 \mid r$,& 	$2$ &	$2$							&	$2$ if $c>1$ or $1$ if $c=1$	\\
		$\tau \notin Z(G)$	&	&	&	\\
		\hline
		$N \rtimes X$, $q^2 \mid r$,	 & $q+1$&	$2$					&	$2$ if $c>1$	or $1$ if $c=1$\\
		$\tau \in Z(G)$	&	&	&	\\
		\hline
		$J_{t,c} \rtimes Y$	& 1	& 1	& 1	\\
		\hline
	\end{tabular}}
	\medskip
	\caption{Results for index $pq$ subgroups of the transitive subgroups of $\Hol(N)$ for $N$ cyclic.}
	\label{cyclic-index-pq}
\end{table}

\subsection{Non-abelian case}\label{metacyclic}
Now let $N \cong C_p \rtimes C_q$. Recall that $e_0$ is the largest power of $q$ dividing $p-1$, and $s=(p-1)/q^{e_0}$. For each transitive subgroup $G\leq \Hol(N)$, we again set $r:=|G|$.
\begin{proposition}\label{metab_conj}
    For the transitive subgroups $G$ in Table \ref{metacyclic-trans-subs}, $\#\mathrm{Cl}_{pq}(G)$ is given by the second column in Table \ref{metab-index-pq}.
	
    All parallel extensions to any extension associated to $G$ admits a Hopf--Galois structure of non-abelian type.
\end{proposition}
\begin{proof}
    Let $G$ be a transitive subgroup of non-abelian type as in Table \ref{metacyclic-trans-subs}. Then for some fixed $c$, fixed $u$ such that $q \nmid u$, and some fixed $d \mid s$, $G$ has the one of the following forms:
    \begin{align*}
        &G_1:=P \rtimes \langle T, A^{q^{e_0-c}}, B^{s/d} \rangle,\\
        &G_2:= P \rtimes \langle TA^{uq^{e_0-c}},B^{s/d} \rangle.
    \end{align*}  
    We recall that, for the index $pq$ subgroups $H$ of $G$, by Proposition \ref{conj_orbs}, we need only look at the conjugacy classes of $H \cap Q$, with $Q=P \rtimes \langle T, A \rangle$, in $G$.

    First, suppose $G=G_1$ for some fixed $c,d$. Up to conjugation by elements of $P$, the index $pq$ subgroups $H$ of $G$ have $H \cap Q$ given by
    \begin{align*}
        &H_1:=\langle a\e_1+\e_2 \rangle \rtimes \langle A^{q^{e_0-c}}\rangle,\\
        &H_2:=\langle \e_i \rangle \rtimes \langle T^xA^{q^{e_0-c}} \rangle,\\
        &H_3:=\langle \e_i \rangle \rtimes \langle T, A^{q^{e_0-(c-1)}} \rangle,
    \end{align*}
    where $i \in \{1,2\}$, $0 \leq x \leq q-1$ and $1 \leq a \leq p-1$. It is clear that families $H_1$, $H_2$ and $H_3$ are pairwise non-conjugate, and no two groups in family $H_2$ with distinct $x$ values are conjugate as $\langle T,A,B \rangle$ is abelian.
    
    Now, as both $\e_1$ and $\e_2$ are eigenvectors of each map in $\langle T,A,B \rangle$, we therefore see that family $H_2$ represents $2q$ conjugacy classes, each of size $1$, and $H_3$ represents $2$ conjugacy classes, also each of size $1$.

    For the groups of type $H_1$, since $\e_1$ and $\e_2$ have the same eigenvalues for $A$ and for $B$, but not for $T$, we need only consider $\langle a\e_1+\e_2 \rangle$ under conjugation by $T^y$, which gives
    \[\langle aa_{\alpha}^{yq^{e_0-1}}\e_1+\e_2 \rangle.\]
    We get a distinct group for each $0 \leq y \leq q-1$, and so this family splits into $(p-1)/q=q^{e_0-1}s$ conjugacy classes, each of size $q$. In total, therefore, we have $q^{e_0-1}s+2q+2$ conjugacy classes.

    Now suppose that $G=G_2$ for some fixed $u,c,d$. Up to conjugation by elements of $P$, the index $pq$ subgroups $H$ of $G_2$ have $H \cap Q$ given by
    \[H_4:=\langle a\e_1+b\e_2 \rangle \rtimes \langle A^{q^{e_0-(c-1)}} \rangle,\]
    where $0 \leq a,b \leq p-1$ with $a,b$ not both $0$. By similar reasoning as above, we see that this splits into $q^{e_0-1}s$ conjugacy classes, each of size $q$ and two conjugacy classes of size $1$.

    We now compute $\mathrm{Core}_{G_j}(H_i)$ for each $1 \leq i \leq 4$ and corresponding $G_j \in \{G_1,G_2\}$.
    
    Firstly, we see that
    \[\mathrm{Core}_{G_1}(H_1)=\bigcap_{G \in G_1}g\langle a\e_1+\e_2\ \rangle \rtimes \left\langle A^{q^{e_0-c}} \right\rangle g^{-1}= \{1\}.\]
    Then for index $pq$ subgroups $H$ of $G_1$ containing $H_1$, and for $M_1:=G_1$, we obtain an isomorphism
    \begin{align*}
        \phi:(G_1,H) &\to (M_1,M_1')\\
        \e_1 &\mapsto a^{-1}\e_1,\\
        \e_2 &\mapsto -\e_2,\\
    \end{align*}
    and acting as the identity on the other generators.
    
    Secondly, we see that
    \[C:=\mathrm{Core}_{G_1}(H_2)=\bigcap_{g \in G_1}g\langle \e_i \rangle \rtimes \left\langle T^xA^{q^{e_0-c}} \right\rangle g^{-1} = \langle \e_i \rangle.\]
    We treat the case $i=1$ here, with the case $i=2$ being similar; for the index $pq$ subgroups $H$ of $G_1$ containing $H_2$, we see:
    \[G_1/C \cong \left\langle \e_2, T, A^{q^{e_0-c}},B^{s/d} \right\rangle\]
    and
    \[H/C \cong \left\langle T^xA^{q^{e_0-c}}, B^{s/d} \right\rangle.\]
    Then for
    \[M_2:=\left\langle \e_2,T,A^{q^{e_0-c}}, B^{s/d} \right\rangle,\]
    we obtain an isomorphism 
    \begin{align*}
        \phi: (G_1/C,H/C) &\to (M_2,M_2'),\\
        A^{q^{e_0-c}} &\mapsto T^{-x}A^{q^{e_0-c}}
    \end{align*}
    and acting as the identity on the other generators. We note that $M_2 \leq \Hol(N)$ is transitive on $N$, but is also isomorphic as a permutation group to the transitive subgroup
    \[\langle \sigma, \tau, \alpha^{q^{e_0-(c-1)}}, \beta \rangle \leq \Hol(C_{pq}),\]
    where $\beta \in \Aut(C_{pq})$ is some automorphism of order $d$ which acts trivially on $\tau$.
    
    Thirdly, we see that
    \[C:=\mathrm{Core}_{G_1}(H_3)=\bigcap_{g \in G_1}g\langle\e_i\rangle \rtimes \left\langle T,A^{q^{e_0-(c-1)}} \right\rangle g^{-1} =
    \begin{cases}
        & \langle \e_1, T \rangle, \; i=1,\\
        & \left\langle \e_2, TA^{-q^{e_0-1}} \right\rangle, \; i=2.
    \end{cases}\]
    Again, we treat the case $i=1$, with the other case being similar; for the index $pq$ subgroups $H$ of $G_1$ containing $H_3$, we see:
    \[G_1/C \cong \langle \e_2, A^{q^{e_0-c}}, B^{s/d} \rangle\]
    and
    \[H/C \cong \langle A^{q^{e_0-(c-1)}}, B^{s/d} \rangle.\]
    Then for
    \[M_3:= \langle \e_2,TA^{q^{e_0-c}}, B^{s/d} \rangle,\]
    we obtain an isomorphism
    \begin{align*}
        \phi: (G_1/C,H/C) &\to (M_3,M_3'),\\
        A^{q^{e_0-c}} &\mapsto TA^{q^{e_0-c}}
    \end{align*}
    and acting as the identity on $\e_2$. We note that $M_3 \leq \Hol(N)$ is transitive on $N$, but is also isomorphic as a permutation group to the transitive subgroup $J_{1,c} \leq \Hol(C_{pq})$.

    Finally, we see that
    \[C:=\mathrm{Core}_{G_2}(H_4)=\bigcap_{g \in G_2}g\langle a\e_1+b\e_2\rangle \rtimes \left\langle A^{q^{e_0-(c-1)}} \right\rangle =
    \begin{cases}
        &\{1\}, \; a,b \neq 0,\\
        &\langle \e_1 \rangle, \; a\neq 0, b=0,\\
        &\langle \e_2 \rangle, \; a=0, b \neq 0.
    \end{cases}\]
    In the case that $C=\{1\}$, for the index $pq$ subgroups $H$ of $G_2$ containing $H_4$, and for $M_4:=G_2$, we get an isomorphism
    \begin{align*}
        \phi:(G_2,H_4) &\to M_4\\
        \e_1 \mapsto a^{-1}\e_1,\\
        \e_2 \mapsto -b^{-1}\e_2
    \end{align*}
    and acting as the identity on the other generators.
    
    In the case that $C=\langle \e_1 \rangle$, we compute:
    \[G_2/C \cong \langle \e_2, TA^{uq^{e_0-c}}, B^{s/d} \rangle\]
    and
    \[H_4/C \cong \langle A^{q^{e_0-(c-1)}}, B^{s/d} \rangle.\]
    We remark that we can directly view $(G_2/C,H/C)$ as a transitive subgroup of $\Hol(N)$. It is also isomorphic as a permutation group to the transitive subgroup $J_{u,c} \leq \Hol(C_{pq})$.

    Therefore, all parallel extensions to any extension associated to a transitive subgroup $G \leq \Hol(N)$ admits a Hopf--Galois structure of non-abelian type. This is because we have been able to find permutation isomorphisms from $(G/C,H/C)$ to some transitive subgroup $M \leq \Hol(N)$ in each case.
\end{proof}
\newpage
\begin{table}
	\centering
	\scalebox{1.1}{
		\begin{tabular}{|c|c|c|c|}
			\hline
			Transitive subgroup	$G$ &	$\#\mathrm{Cl}_{pq}(G)$	&	$\#\text{Aut}(G)$-orbits	&	$\#$Isom. classes\\
			\hline
			$P\rtimes \langle T,A^{q^{e_0-c}},B^{s/d}\rangle$	& 	$q^{e_0-1}s+2q+2$ &	$3(\varphi(q^c)+2)/2$ if $c\geq 1$, &  $2$ if $c\geq 1$,\\
			& & $2$ if $c=0$, & $1$ if $c=0$.\\
			\hline
			$P\rtimes \langle TA^{uq^{e_0-c}},B^{s/d}\rangle$	&	$q^{e_0-1}s+2$ &	$2$ if $(c,u)=(1,\frac{1}{2}(q-1))$, & $1$ \\
			&				   &	$3$ otherwise.						  & \\
			\hline
	\end{tabular}}
	\medskip
	\caption{Results for index $pq$ subgroups of the transitive subgroups of $\Hol(N)$ for $N$ non-abelian.}
	\label{metab-index-pq}
\end{table}
We note that $\varphi$ in Table \ref{metab-index-pq} denotes the Euler totient function.
\begin{proposition}\label{metab_orbit_1}
    Let
    \[G=P\rtimes \langle T,A^{q^{e_0-c}},B^{s/d}\rangle\]
    be as in row one of Table \ref{metacyclic-trans-subs}. Then the number of $\Aut(G)$-orbits and abstract isomorphism types of the index $pq$ subgroups $H$ are given by the third columns in Table \ref{metab-index-pq}.
\end{proposition}
\begin{proof}
    For the index $pq$ subgroups $H$ of $G$, by Proposition \ref{conj_orbs}, we need only consider that $\Aut(G)$-orbits of $H \cap Q$, where $Q:= P \rtimes \langle T,A\rangle$.

    We will consider the following three cases: (i) $c=0$, (ii) $c=1$ and (iii) $c>1$.
	
    (i) If
    \[G=P \rtimes \langle T,B^{s/d} \rangle,\]
    then the index $pq$ subgroups $H$ of $G$ have $H \cap Q$ of the forms:
    \begin{align*}
        &H_1:= \langle \e_1 \rangle,\\
        &H_2:= \langle \e_2 \rangle,\\
        &H_3:= \langle \e_1+b\e_2 \rangle, \; 1 \leq b \leq p-1.
    \end{align*}
     Each subgroup is cyclic of order $p$. We now wish to compute their orbits under $\Aut(G)$. To this end, it can be seen that if $\phi \in \Aut(G)$, then either 
     \[\phi(\e_1)=x_1\e_1,\phi(\e_2)=y_1\e_2 \text{ with } 1\leq x_1,y_1 \leq p-1,\]
     or
     \[\phi(\e_1)=x_2\e_2,\phi(\e_2)=y_2\e_1 \text{ with } 1 \leq x_2,y_2 \leq p-1.\]
     We thus obtain two orbits under $\Aut(G)$, one represented by $\{H_1,H_2\}$ and one represented by $\{H_3\}$.
	
    (ii) If
    \[G= P \rtimes \langle T, A^{q^{e_0-1}},B^{s/d} \rangle.\]
    then, up to conjugation under $P$, the index $pq$ subgroups $H$ of $G$ have $H \cap Q$ of the forms:
    \begin{align*}
        &H_1:=\left\langle \e_i,T \right\rangle \cong
        \begin{cases}
            & C_{pq}, \; i=1,\\
            & C_p \rtimes C_q, \; i=2,
        \end{cases}\\
        &H_2:= \left\langle \e_i, TA^{-q^{e_0-1}}\right\rangle \cong
        \begin{cases}
            & C_p \rtimes C_q, \; i=1,\\
            & C_{pq}, \; i=2,
        \end{cases}\\
        &H_3:= \left\langle \e_i,A^{q^{e_0-1}} \right\rangle \cong C_p \rtimes C_q,\\
        &H_4:= \left\langle \e_1+\mu\e_2,A^{q^{e_0-1}} \right\rangle \cong C_p \rtimes C_q,\\
        &H_5:= \left\langle \e_i,TA^{\nu q^{e_0-1}}\right\rangle \cong C_p \rtimes C_q,\\
        &H_6:= \left\langle \e_1+\mu\e_2,TA^{\nu q^{e_0-1}}\right\rangle \cong C_p \rtimes C_q,
    \end{align*}
    where $i \in \{1,2\}$, $1 \leq \mu \leq p-1$ and $1 \leq \nu \leq q-2$.
    
    We now wish to compute their orbits under $\Aut(G)$. Therefore, if $\phi \in \Aut(G)$, then, up to composition with conjugation under $P$, $\phi$ belongs to one of the following two families:\\
    Family 1:
    \[\begin{array}{ll}
    \phi_1(\e_1)=x_1\e_1,	&	\phi_1(\e_2)=y_1\e_2,\\
    \phi_1(T)=T,            &   \phi_1(A^{q^{e_0-1}})=A^{q^{e_0-1}}.
    \end{array}\]
    for $1 \leq x_1,y_1 \leq p-1$.\\
    Family 2:
    \[\begin{array}{ll}
    \phi_2(\e_1)=x_2\e_2,	&	\phi_2(\e_2)=y_2\e_1,\\
    \phi_2(T)=T^{-1}A^{q^{e_0-1}}, & \phi_2(A^{q^{e_0-1}})=A^{q^{e_0-1}}.
    \end{array}\]
    for $1 \leq x_2,y_2 \leq p-1$.

    We first note that applying $\phi_1$ allows us to assume that $\mu=1$ in the subgroups of type $H_4$ and $H_6$. We thus now understand the orbits under the automorphisms of type $\phi_2$.

    To this end, we see that $\langle \e_2,T \rangle$ is in the same orbit as $\langle \e_1,TA^{-q^{e_0-1}} \rangle$, and $\langle \e_1,T \rangle$ is in the same orbit as $\langle \e_2,TA^{-q^{e_0-1}} \rangle$. We also see that $H_3$ forms a single distinct orbit, as does $H_4$. We see that $\phi_2$ takes the subgroup $\langle \e_1,TA^{\nu q^{e_0-1}} \rangle$ to the subgroup $\langle \e_2,TA^{-(1+\nu)q^{e_0-1}} \rangle$, and vice-versa. Thus $H_5$ splits into $q-2$ distinct orbits. Finally, the subgroup $\langle \e_1+\e_2,TA^{\nu q^{e_0-1}} \rangle$ is taken to $\langle \e_1+\e_2,TA^{-(1+\nu) q^{e_0-1}} \rangle$ and vice versa. We therefore see that $H_6$ splits into $(q-1)/2$ distinct orbits.

    (iii) Finally, if
    \[G= P \rtimes \langle T, A^{q^{e_0-c}},B^{s/d} \rangle.\]
    then, up to conjugation under $P$, the index $pq$ subgroups $H$ of $G$ have $H \cap Q$ of the forms:
    \begin{align*}
        &H_1:= \left\langle \e_i, A^{q^{e_0-c}}\right\rangle \cong C_p \rtimes C_{q^c},\\
        &H_2:= \left\langle \e_i, TA^{\nu q^{e_0-c}}\right\rangle \cong C_p \rtimes C_{q^c},\\
        &H_3:= \left\langle \e_i, T,A^{q^{e_0-(c-1)}} \right\rangle \cong C_{pq} \rtimes C_{q^{c-1}},\\
        &H_4:= \left\langle \e_1+\mu\e_2,A^{q^{e_0-c}}\right\rangle \cong C_p \rtimes C_{q^c},\\
        &H_5:= \left\langle \e_1+\mu\e_2,TA^{\nu q^{e_0-c}}\right\rangle \cong C_p \rtimes C_{q^c},
    \end{align*}
    where $i \in \{1,2\}$, $1 \leq \mu \leq p-1$ and $\nu \in \mathbb{Z}_{q^c}^{\times}$.
    
    We now wish to compute their orbits under $\Aut(G)$. Therefore, if $\phi \in \Aut(G)$, then, up to composition with conjugation under $P$, $\phi$ belongs to one of the following two families:\\	
    Family 1:
    \[\begin{array}{ll}
    \phi_1(\e_1)=x_1\e_1,	&	\phi_1(\e_2)=y_1\e_2,\\
    \phi_1(T)=T,            &   \phi_1(A^{q^{e_0-c}})=A^{q^{e_0-c}},
    \end{array}\]
    for $1 \leq x_1,y_1 \leq p-1$.\\
    Family 2:
    \[\begin{array}{ll}
    \phi_2(\e_1)=x_2\e_2,	&	\phi_2(\e_2)=y_2\e_1,	\\
    \phi_2(T)=T^{-1}A^{q^{e_0-1}}, & \phi_2(A^{q^{e_0-c}})=A^{q^{e_0-c}},
    \end{array}\]
    for $1 \leq x_2,y_2 \leq p-1$.

    Once again, we note that applying $\phi_1$ allows us to assume that $\mu=1$ in the subgroups of type $H_4$ and $H_5$. We thus now want to understand the orbits under the automorphisms of type $\phi_2$.

    To this end, we see that $H_1$, $H_3$, and $H_4$ each form distinct orbits. Next, the subgroup $\langle \e_i, TA^{\nu q^{e_0-c}} \rangle$ is sent to the subgroup $\langle \e_{3-i},TA^{-(1+\nu)q^{e_0-c}} \rangle$, so $H_2$ splits into $\varphi(q^c)$ distinct orbits. Finally, the subgroup $\langle \e_1+\e_2,TA^{\nu q^{e_0-c}} \rangle$ is sent to the subgroup $\langle \e_1+\e_2,TA^{-(1+\nu)q^{e_0-c}} \rangle$. We therefore see that $H_5$ splits into $\varphi(q^c)/2$ distinct orbits.

    To summarise, for $c=0$, there are two $\Aut(G)$-orbits, for $c=1$, there are $3(q+1)/2=3(\varphi(q)+2)/2$ $\Aut(G)$-orbits, and for $c>1$, there are $3(\varphi(q^c)+2)/2$ $\Aut(G)$-orbits.
\end{proof}

\begin{proposition}\label{metab_orbit_2}
    Let
    \[G=P\rtimes \langle TA^{uq^{e_0-c}},B^{s/d}\rangle,\]
    be as in row two of Table \ref{metacyclic-trans-subs}. Then the number of $\Aut(G)$-orbits and abstract isomorphism types of the index $pq$ subgroups $H$ are given by the fourth column in Table \ref{metab-index-pq}.
\end{proposition}
\begin{proof}
    For the index $pq$ subgroups $H$ of $G$, by Proposition \ref{conj_orbs}, we need only consider the $\Aut(G)$-orbits of $H \cap Q$, where $Q:=P \rtimes \langle T,A \rangle$.

    Up to conjugation by $P$, the index $pq$ subgroups $H$ of $G$ have $H \cap Q$ of the forms
    \begin{align*}
        &H_1:=\left\langle \e_i,A^{q^{e_0-(c-1)}}\right\rangle \cong C_p \rtimes C_{q^{c-1}},\\
        &H_2:=\left\langle \e_1+\mu\e_2,A^{q^{e_0-(c-1)}}\right\rangle \cong C_p \rtimes C_{q^{c-1}}
    \end{align*}
    where $i \in \{1,2\}$ and $1 \leq \mu \leq p-1$.
    
    We now wish to compute their orbits under $\Aut(G)$. Therefore, if $\phi\in\Aut(G)$, then, up to composition with conjugation under $P$, $\phi$ belongs to one of the following two families:

    Family 1:
    \[\begin{array}{ll}
        \phi_1(\e_1)=x_1\e_1,	&	\phi_1(\e_2)=y_1\e_2,\\
        \phi_1(A^{q^{e_0-(c-1)}})=A^{q^{e_0-(c-1)}}, &
    \end{array}\]
    for $1 \leq x_1,y_2 \leq p-1$.

    In the case that $(c,u)=(1,\frac{1}{2}(q-1))$, we additionally have family 2:
    \[\begin{array}{ll}
        \phi_2(\e_1)=x_2\e_2,	&	\phi_2(\e_2)=y_2\e_1.
    \end{array}\]
    We first note that applying $\phi_1$ allows us to assume that $\mu=1$ in the subgroups of type $H_2$. Thus, whenever $(c,u)\neq(1,\frac{1}{2}(q-1))$, we get three distinct orbits. In the special case that $(c,u)=(1,\frac{1}{2}(q-1))$, we now note that the subgroups of type $H_1$ form a single orbit, giving two distinct orbits in total.
\end{proof}
Propositions \ref{cyclic-conjugacy} and \ref{metab_conj} give the following theorem:
\begin{theorem}\label{parallel_HGS}
    Let $L/K$ be a separable extension of degree $pq$ admitting a Hopf--Galois structure of type $N$. Then all parallel extensions to $L/K$ also admit a Hopf--Galois structure of type $N$. 
\end{theorem}
\begin{proof}
    Proposition \ref{cyclic-conjugacy} gives the statement for $N=C_{pq}$ and Proposition \ref{metab_conj} gives the statement for $N=C_p \rtimes C_q$ for the transitive subgroups $G \leq \Hol(N)$ with $|G|$ divisible by $q^2$. Lemma 3.22 of \cite{Dar24(a)} tells us that every transitive subgroup $G \leq \Hol(C_p \rtimes C_q)$ with $q^2 \nmid |G|$ is isomorphic, as a permutation group, to a transitive subgroup of $\Hol(C_{pq})$. Thus given such a subgroup, we can then translate the discussion to viewing $G$ as a transitive subgroup of $\Hol(C_{pq})$. Proposition \ref{cyclic-conjugacy} then tells us that all index $pq$ subgroups of $G$ correspond to extensions admitting a Hopf--Galois structure of cyclic type. Then a further application of Lemma 3.22 of \cite{Dar24(a)} tells us that the associated extension also admits a Hopf--Galois structure of non-abelian type.
\end{proof}

\section{Parallel extensions admitting no Hopf--Galois structures}
We will say that a separable extension $L/K$ admits the \textit{parallel no-HGS} property if $L/K$ admits a Hopf--Galois structure, and has a parallel extension $L'/K$ admitting no Hopf--Galois structure of any type. Theorem \ref{parallel_HGS} tells us that no separable extension of degree $pq$ admits the parallel no-HGS property. Thus the question remains as to whether there are indeed examples of extensions admitting the parallel no-HGS property, and if so, how rare such a phenomenon is.

In order to investigate this question for extensions $L/K$ of a fixed degree $n$, it is necessary to know a lot about the transitive subgroups of $\Hol(N)$ for each abstract group $N$ of order $n$. More specifically, if $L/K$ has Galois closure $E/K$ of degree $mn$, answering the question requires a classification of all transitive subgroups of $\Hol(N)$ (for each $N$ of order $n$) of order divisible by $n$ and dividing $mn$.

Our approach to searching for extensions of low degree $n$ admitting the parallel no-HGS property has been to write a computer algorithm, see \cite{Dar24(b)}, which can be implemented in MAGMA, \cite{BCP97}. The algorithm first builds a list of all transitive subgroups $G$ of $\Hol(N)$ for each group $N$ of order $n$. For each $G$, it then computes all subgroups $H$ of order $|G|/n$, followed by computing the core $C:=\mathrm{Core}_G(H)$. We then consider the natural homomorphism $\pi:G \to J:=G/C$ and view $H/C$ as the image $J'$ of $H$ under $\pi$ (this makes sure Magma understands that $H/C$ is a genuine subgroup of $G/C$). We then ask whether $J$ is isomorphic as an abstract group to any transitive subgroup $M$ of some $\Hol(N')$ (for some group $N'$ of order $n$) in our original list. If such an isomorphism $\phi$ exists, we can construct all isomorphisms from $J$ to $M$ by composing $\phi$ with elements of $\Aut(M)$. We finally check whether any of these isomorphisms send $J'$ to $M':=\mathrm{Stab}_M(1_{N'})$. If no isomorphism exists, then given a separable extension $L/K$ with Galois closure $E$ such that $\Gal(E/K) \overset{\phi}{\cong} G$ with $\phi(\Gal(E/L))= G':=\mathrm{Stab}_G(1_N)$, we have that $L/K$ admits the parallel no-HGS property. If such a transitive subgroup $G$ is found, then we also say that $G$ admits the parallel no-HGS property.

The algorithm described has enabled us to find examples of transitive subgroups $(G,G')$ of degrees $n=8,12,24$ and $27$ which admit the parallel no-HGS property.

It is to be remarked, however, that the occurrence of the parallel no-HGS property appears to be quite rare. To get a sense of this, we present Table \ref{par_no_HGS}, which demonstrates how rare the phenomenon is even within the families (characterised by the degree) of extensions within which it does occur. The first column gives the degree $n$ of the extensions in question, the second column gives the sum of the number of transitive subgroups of $\Hol(N)$ up to conjugacy for each abstract group $N$ of order $n$, and the final column gives the number of such subgroups admitting the parallel no-HGS property.
\begin{table}
    \centering
    \begin{tabular}{|c|c|c|}
    \hline
        Degree & \#Trans. sbgps & \#Parallel no-HGS\\
        \hhline{|=|=|=|}
        8  & 148 & 8\\
        \hline
        12 & 134 & 23\\
        \hline
        24 & 4752 & 396 \\
        \hline
        27 & 739 & 163 \\
        \hline
    \end{tabular}
    \caption{Number of permutation groups admitting the parallel no-HGS property}
    \label{par_no_HGS}
\end{table}
We present a degree 8 example of the occurrence of the parallel no-HGS property in the form of a theorem:
\begin{theorem}\label{no_HGS}
    There is a separable extension $L/K$ of degree $8$ admitting the parallel no-HGS property.
    
    Further, all parallel extensions to any extension $F/K$ of degree $8$ admitting the parallel no-HGS property have the same Galois closure as $F/K$.
\end{theorem}
\begin{proof}
    We give an example of a transitive subgroup $G \leq \Hol(C_2 \times C_4)$ with the parallel no-HGS property.

    Consider the following permutation representation of $\Hol(C_2 \times C_4)$:
    \begin{align*}
        \Hol(C_2 \times C_4) = \langle &(1, 5, 2, 6)(3, 7, 4, 8), (1, 3)(2, 4)(5, 7)(6, 8),\\
                            &(1, 2)(3, 4)(5, 6)(7, 8), (5, 7)(6, 8),(3, 4)(7, 8)\\
                            &(5, 6)(7, 8)\rangle.
    \end{align*}
    One may identify $C_2 \times C_4$ inside $\Hol(C_2 \times C_4)$ with the subgroup
    \begin{align*}
        C_2 \times C_4 \cong \langle    & (1, 5, 2, 6)(3, 7, 4, 8),\\
                                        & (1, 3)(2, 4)(5, 7)(6, 8)\rangle.
    \end{align*}
    Then the following group $G$ of order $32$, given by
    \begin{align*}
        G := \langle  &(1, 5, 4, 8)(2, 6, 3, 7), (1, 5, 2, 6)(3, 8, 4, 7),\\
	               &(1, 3)(2, 4)(5, 7)(6, 8),(1, 2)(3, 4)(5, 6)(7, 8), (5, 6)(7, 8) \rangle
                      &\cong (C_2 \times C_2) \cdot D_4 \text{ (non-split extension)}
    \end{align*}
    is a transitive subgroup of $\Hol(C_2 \times C_4)$, with $G':=\mathrm{Stab}_G(1)$ of order $4$ given by
    \[G' = \langle  (5, 6)(7, 8),(3, 4)(5, 7, 6, 8) \rangle.\]
    The group
    \[H=\langle (5, 6)(7, 8), (1, 3)(2, 4)(5, 8)(6, 7) \rangle \cong C_2 \times C_2\]
    is an index $8$ subgroup of $G$, with trivial core.

    Then it can be exhaustively checked by computer that, for any group $N$ of order $8$ and any transitive subgroup $M \leq \Hol(N)$, there is no isomorphism $\phi:G \to M$ such that $\phi(H)=\mathrm{Stab}_M(1_N)$.

    Therefore, a separable extension $L/K$ with Galois closure $E$ such that\\ $\Gal(E/K) \overset{\psi}{\cong} G$ with $\psi(\Gal(E/L))= G'$ admits the parallel no-HGS property. In particular, $L/K$ admits a Hopf--Galois structure of type $C_2 \times C_4$, but has a parallel extension which admits no Hopf--Galois structure of any type.

    Every such degree 8 example given by the algorithm gives an index $8$ subgroup with trivial core, hence all parallel extensions of interest have the same Galois closure as the original extension.
\end{proof}
We remark that we have not yet found an example of the phenomenon for separable extensions of squarefree degree. This observation, along with the discussion in Section \ref{par_pq} inspires the following conjecture.
\begin{conjecture}
    There is no separable extension of squarefree degree admitting the parallel no-HGS property.
\end{conjecture}
Despite the rare occurrence of the phenomenon, given an extension admitting the parallel no-HGS property, we now give a way of constructing two infinite families of permutation groups admitting the property. At present, we can only do so for examples of odd degree, but it should be possible to extend this to examples of any degree. We start with the following lemmas:
\begin{lemma}\label{hom_proj}
    Let $N$ be a group and $M$ a characteristic subgroup of $N$. Then there is a group homomorphism $\phi:\Hol(N) \to \Hol(N/M)$ given by
    \[\phi(\eta,\alpha)=(\overline{\eta},\overline{\alpha})\]
    where $\overline{\eta}:=\eta M$ and $\overline{\alpha}(\overline{\eta}):=\alpha(\eta)M$.
\end{lemma}
We note that if $G\leq \Hol(N)$ is transitive on $N$, then $\phi(G)$ is transitive on $N/M$.
\begin{proof}
    We show that $\overline{\alpha}$ is well-defined. Take $\eta \in N, \mu \in M$, then
    \[\alpha(\eta\mu)M=\alpha(\eta)\alpha(\mu)M=\alpha(\eta)M.\]
    The second equality is due to the fact $M$ is characteristic in $N$, and so $\alpha(\mu) \in M$.
\end{proof}
\begin{lemma}\label{gcd}
    For any integer $n$, there are infinitely many primes $q$ such that $\gcd(q-1,n)\leq 2$.
\end{lemma}
\begin{proof}
    By the Chinese Remainder Theorem, the system of congruences
    \begin{equation}\label{system}
        x-1 \equiv 2 \bmod{4}, \;\; x-1 \equiv 1 \bmod{p}, \;\; \text{ for all odd primes } p \mid n
    \end{equation}
    has infinitely many solutions. Define $\mathrm{rad}_{\mathrm{odd}}(n)$ to be the product of all distinct odd primes dividing $n$. Then given a solution $x$ of (\ref{system}), for any integer $k$, we have that
    \[x'=x+4k\mathrm{rad}_{\mathrm{odd}}(n)\]
    is also a solution of (\ref{system}). Next, given that $x$ and $4\mathrm{rad}_{\mathrm{odd}}(n)$ are coprime, Dirichlet's Theorem tells us that there are infinitely many primes $q$ of the form
    \[q=x+4k\mathrm{rad}_{\mathrm{odd}}(n).\]
    If $n$ is odd, then $q-1 \equiv x-1 \equiv 1 \bmod{p}$ for all primes $p \mid n$, and so $\gcd(q-1,n)=1$. If $n$ is even, then $q-1$ is coprime to the largest odd factor of $n$, and so $q-1 \equiv x-1 \equiv 2 \bmod{4}$ forces $\gcd(q-1,n)=2$.
\end{proof}
\begin{proposition}\label{infinite}
    Let $G$ be a transitive subgroup of $\Hol(N)$ for some group $N$ of order $n$ with $n$ odd such that $G$ admits the parallel no-HGS property. Then there are infinitely many primes $q$ such that $M \cong G \times C_q \leq \Hol(N \times C_q)$ is a transitive subgroup also admitting the parallel no-HGS property.
\end{proposition}
As demonstrated by Table \ref{par_no_HGS}, we have found examples of transitive subgroups of degree 27 admitting the parallel no-HGS property, and so it is reasonable to ask that $n$ is an odd integer.
    \begin{proof}
As $G$ admits the parallel no-HGS property, there is an index $n$ subgroup $H$ of $G$ with $C:=\mathrm{Core}_G(H)$ such that $(G/C,H/C)$ is not isomorphic as a permutation group to any transitive subgroup of $\Hol(N')$ for any group $N'$ of order $n$.

    By Lemma \ref{gcd}, there are infinitely many primes $q$ such that
    \begin{enumerate}[label=(\roman*)]
        \item $\gcd(q-1,n)=1$,
        \item $q>n$,
        \item $q>|\Aut(Y)|$ for every group $Y$ of order $n$.
    \end{enumerate}
    In particular, if we ask that $q>n!$, we get both (ii) and (iii). So, for the rest of this proof, we let $q>n!$ be a prime satisfying $\gcd(q-1,n)=1$.

    We note that $G \times C_q$ is a transitive subgroup of $\Hol(N \times C_q)$ with $\mathrm{Stab}_{G \times C_q}(1_N)=G' \times \{1\}$, where $G':=\mathrm{Stab}_G(1_N)$. One index $qn$ subgroup of $G \times C_q$ is $H \times \{1\}$, which has $\mathrm{Core}_{G \times C_q}(H \times \{1\})=C \times \{1\}$. We then have $(G \times C_q)/(C \times \{1\}) \cong G/C \times C_q$ and $(H \times \{1\})/(C \times \{1\}) \cong H/C \times \{1\}$. We will show that $(G/C \times C_q,H/C \times \{1\})$ is not isomorphic as a permutation group to any transitive subgroup $I \leq \Hol(J)$ for any group $J$ of order $qn$. We will aim to get contradiction, assuming that we have such a situation.
    
    By (ii), $J$ has a unique Sylow $q$-subgroup, which is therefore characteristic. Then by the Schur-Zassenhaus theorem, $J$ must be a split extension of $C_q$. In particular, we have that $J=C_q \rtimes Y$ for some group $Y$ of order $n$. Finally, by (i), the action of $Y$ on $C_q$ is trivial, and hence we can write $J=C_q \times Y$. We note that $Y$ is also characteristic in $J$, and hence $\Hol(J)=\Hol(C_q) \times \Hol(Y)$.

    Now, suppose we have a transitive subgroup $I \leq \Hol(J)$ with $(I,I')$ isomorphic to $(G/C \times C_q,H/C \times \{1\})$, where $I':=\mathrm{Stab}_I(1_J)$. By Lemma \ref{hom_proj}, we have a homomorphism $\phi: \Hol(J) \to \Hol(Y)$, and that $\phi(I)$ is a transitive subgroup of $\Hol(Y)$. By (ii) and (iii), the factor $C_q$ must have trivial image in $\phi$, and so $G/C$ must act transitively (but a priori not necessarily faithfully) on $Y$ with stabiliser $H/C$. The kernel of this action is then the core of $H/C$ in $G/C$, which is trivial since $C=\mathrm{Core}_G(H)$. Therefore $(G/C,H/C)$ embeds as a transitive subgroup of $\Hol(Y)$, giving a contradiction.
\end{proof}
 \begin{corollary}\label{infty_cor}
    Let $G$ be a transitive subgroup of $\Hol(N)$ for some group $N$ of odd order $n$ such that $G$ admits the parallel no-HGS property. Then for any positive integer $k$, there are infinitely many sequences $p_1,\cdots,p_k$ of length $k$ of primes coprime to $n$ such that $M \cong G \times C_m \leq \Hol(N \times C_m)$ is a transitive subgroup also admitting the parallel no-HGS property, where $m=p_1\cdots p_k$.
 \end{corollary}
 \begin{proof}
     This follows from repeated application of Proposition \ref{infinite}.
 \end{proof}

\section*{Acknowledgements}
This paper was completed under the support of the following two grants:

The Engineering and Physical Sciences Doctoral Training Partnership research grant (EPSRC DTP).

Project OZR3762 of Vrije Universiteit
Brussel and FWO Senior Research Project G004124N.
\bibliography{MyBib}

\end{document}